\documentclass[12pt]{extarticle}
\usepackage{amsmath, amsthm, amssymb, hyperref, color}
\usepackage{graphicx}
\usepackage{caption}
\usepackage{mathtools}
\usepackage{enumerate}
\usepackage[all]{xypic}
\usepackage{verbatim}
\usepackage{chemfig, chemnum}
\usepackage{tikz}
\usepackage[linesnumbered,lined,commentsnumbered,ruled]{algorithm2e}
\usepackage{caption}
\usepackage{todonotes}

\usepackage{subcaption}
\oddsidemargin \evensidemargin
\usepackage{makecell}
\usepackage{array}
\newcolumntype{?}{!{\vrule width 1pt}}

\theoremstyle{definition}
\newtheorem{theorem}{Theorem}
\newtheorem{proposition}[theorem]{Proposition}
\newtheorem{lemma}[theorem]{Lemma}
\newtheorem{corollary}[theorem]{Corollary}
\theoremstyle{definition}

\newtheorem{remark}[theorem]{Remark}

\newtheorem{example}[theorem]{Example}

\numberwithin{theorem}{section}

\newcommand{\RR}{\mathbb{R}}

\newcommand{\CC}{\mathbb{C} }

\newcommand{\GL}{\operatorname{GL}}
\newcommand{\SL}{\operatorname{SL}}
\newcommand{\SO}{\operatorname{SO}}
\newcommand{\lie}{\mathfrak{g}}

\DeclareMathOperator{\rank}{rank}

\DeclareMathOperator{\Sing}{Sing}

\DeclareMathOperator{\ann}{Ann}

\DeclareMathOperator{\diag}{\text{diag}}

\title{\textbf{Algebraic Analysis of Rotation Data}}

\author{Michael F. Adamer, Andr\'{a}s C. L{\H o}rincz, \\ Anna-Laura Sattelberger, and Bernd Sturmfels}
\date{}

\begin{document}
\maketitle

\begin{abstract}
\noindent 
We develop algebraic tools for statistical inference from
samples of rotation matrices. This rests on the theory of
$D$-modules in algebraic analysis.
Noncommutative Gr\"obner bases are used to
design numerical algorithms for maximum likelihood estimation,
 building  on the holonomic gradient method
of Sei, Shibata, Takemura, Ohara, and Takayama.
We study the Fisher model for sampling from rotation matrices,
and we apply our algorithms for data from the applied sciences.
On the theoretical side, we generalize
the underlying equivariant $D$-modules from ${\rm SO}(3)$
to arbitrary Lie groups. For compact groups, our $D$-ideals
encode the normalizing constant of the Fisher model.
\end{abstract}

\section{Introduction}

Many of the multivariate functions that arise in statistical inference are holonomic.
Being holonomic roughly means that the function is annihilated by a system of linear partial differential
operators with polynomial coefficients whose solution space is finite-dimensional.
Such a system of PDEs can be written as a left ideal in the Weyl algebra, or $D$-ideal, for short.
This representation allows for the application of algebraic geometry
and algebraic analysis, including the use of computational tools, 
such as Gr\"obner bases in the Weyl algebra
~\cite{satstu, Tak}.
 
 This important connection between statistics and algebraic analysis was
 first observed by a group of scholars in Japan, and it led to their
 development of the {\em Holonomic Gradient Method} (HGM)
 and the {\em Holonomic Gradient Descent} (HGD). We
 refer to~\cite{HNTT, koyama, hgmR}
 and to further references given therein.
 The point of departure for the present article is the work
 of Sei et al.~\cite{sei}, who developed HGD for data sampled from the
 rotation group ${\rm SO}(n)$,
 and the article of Koyama \cite{koyama} who undertook a study of
 the associated equivariant $D$-module.
 
  The statistical model we examine in this article
 is the Fisher distribution on the group of rotations,
  defined in (\ref{eq:fisherdensity}) and
 (\ref{eq:normalizing1}).  The aim of
 maximum likelihood estimation (MLE)
 is to learn the model parameters $\,\Theta\,$
 that best explain a given data set.
 In our case,   the MLE problem is difficult because
 there is no simple formula for evaluating
 the normalizing constant of the distribution.
 This is where algebraic analysis comes in.
 The normalizing constant is a holonomic function
 of the model parameters, and we can use
 its holonomic $D$-ideal to derive
 an efficient numerical scheme for solving the
 maximum likelihood estimation  problem.
 
 The present paper is organized as follows.
 Section~\ref{sec2} is purely expository.
 Here, we introduce the Fisher model, and we express
 its log-likelihood function in terms of the sufficient statistics of the given data. These are
 obtained  from the singular value decomposition of  the sample mean.
 In Section~\ref{sec:holonomic}, we turn to algebraic analysis.
 We review the holonomic $D$-ideal in \cite{sei} that annihilates the normalizing
 constant of the Fisher distribution,
 and we derive its associated Pfaffian system.
 Passing to $n\geq 3$, we next
   study  the $D$-ideals on ${\rm SO}(n)$ given in~\cite{koyama}.
 First new results can be found in
 Theorem \ref{thm:inv} and in Propositions \ref{prop:rank4} 
 and \ref{thm:inclusions}.
 
Section~\ref{HGDinaction} is concerned with
 numerical algorithms for maximum likelihood  estimation.
 We develop and compare
     Holonomic Gradient Ascent (HGA),
 Holonomic BFGS (H-BFGS) and
a Holonomic Newton method.
We implemented these methods in the language {\tt R}.
Section~\ref{sec5} highlights how samples of rotation matrices arise
in the sciences and engineering. Topics range from materials science 
and geology to astronomy and biomechanics. We apply holonomic methods 
to data from
the literature, and we discuss both successes and challenges.

The $D$-ideal of the normalizing constant  is of independent interest 
from the perspective of representation theory,
as it generalizes naturally to other Lie groups.
The development of that theory is our  main 
new mathematical contribution. This work is presented in Section~\ref{sec6}.

\section{The Fisher model for random rotations}\label{sec2}

 In this section, we introduce the Fisher model on the rotation group, building on~\cite{sei}.
The group $\,{\rm SO}(3)\,$ consists of all real $\,3 \times 3\,$ matrices $\,Y\,$
that satisfy $\,Y^{\text{t}} Y \,=\, {\rm Id}_3\,$ and ${\rm det}(Y) \,=\, 1$.
This is a smooth algebraic variety of dimension $3$ in the
$9$-dimensional space $\,\RR^{3 \times 3}$.
See~\cite{sonpaper} for a study of rotation groups from the perspective
of combinatorics and algebraic geometry.

The Haar measure on $\,{\rm SO}(3)\,$ is
the unique probability measure $\,\mu\,$ that is invariant under the group action.
The {\em Fisher model} is a family of probability distributions on
$\,{\rm SO}(3)\,$ that is parametrized by $\,3 \times 3\,$ matrices $\,\Theta $.
For a fixed $\Theta$, the density of the {\em Fisher distribution}~equals
\begin{equation}
\label{eq:fisherdensity}
 \qquad \qquad
f_\Theta(Y) \,\,= \,\,\frac{1}{c(\Theta)}\cdot \exp(\text{tr}\left( \Theta^{\text{t}}\cdot Y\right)) \qquad 
\hbox{for all} \,\,\,\, Y\,\in\, {\rm SO}(3).
\end{equation}
This is the density with respect to Haar measure $\mu$.
The denominator is the {\em normalizing constant}. It is chosen such that
$\,\int_{{\rm SO}(3)} f_\Theta(Y) \mu(dY) \,=\, 1$.
This requirement is equivalent to 
\begin{equation}
\label{eq:normalizing1}
c(\Theta) \,\,= \,\,\int_{{\rm SO}(3)}\!\! \exp(\textrm{tr}(\Theta^\text{t}\cdot Y))\mu(dY).
\end{equation}
This function is the Fourier--Laplace transform of the Haar measure~$\mu$; see
Remark~\ref{rmk:fourier}.
The Fisher model is an exponential family. It is one of the
simplest statistical models on $\SO(3)$. 
The task at hand is the accurate numerical evaluation of the integral 
(\ref{eq:normalizing1}) for given $\,\Theta\,$ in~$\,\RR^{3 \times 3}$.
We begin with the observation that, since  integration is
against the Haar measure, the function (\ref{eq:normalizing1}) is invariant under
multiplying $\,\Theta\,$ on the left or right by a rotation matrix:
$$ \qquad \qquad c ( Q\cdot \Theta\cdot R) \,\, = \,\, c(\Theta) \qquad \hbox{for all } \,\, Q,R \,\in\, {\rm SO}(3).$$
In order to evaluate (\ref{eq:normalizing1}), we can therefore restrict to the case of diagonal matrices.
Namely, given any $\,3 \times 3\,$ matrix $\,\Theta$, we first compute its 
{\em sign-preserving singular value decomposition}
$$ \Theta \,\, = \,\,  Q \cdot {\rm diag}(x_1,x_2,x_3) \cdot R . $$ 
Sign-preserving means that $Q,\,R\in\SO(3)$ and  $|x_1| \geq x_2 \geq x_3 \geq 0$.
For non-singular $\Theta$ this implies that $\,x_1 > 0$ whenever ${\rm det}(\Theta) > 0\,$
and $\,x_1 <  0\,$ otherwise.

The normalizing constant $\,c(\Theta)\,$ is the
following function of the three singular~values:
\begin{equation}
\label{eq:normalizing2}
\tilde c(x_1,x_2,x_3) \,\, \,\coloneqq\,  \,\, c({\rm diag}(x_1,x_2,x_3))\,\,=\,\,
\int_{{\rm SO}(3)}\!\!\! \exp(x_1 y_{11} + x_2 y_{22} + x_3 y_{33}) \mu(dY). 
\end{equation}

The statistical problem we address in this paper is parameter estimation for the Fisher model.
Suppose we are given a finite sample $\,\{Y_1, Y_2,\ldots,Y_N\}\,$
from the rotation group $\,\text{SO}(3)$.
We refer to Figure~\ref{fig:MedData} for a concrete example.
Our aim is to find the parameter
matrix $\,\Theta\,$ whose Fisher distribution $\, f_\Theta\,$ best explains the  data.
We work in the classical framework of likelihood inference, i.e.,~we seek
to compute the maximum likelihood estimate (MLE)
for the given data $\,\{Y_1, Y_2,\ldots,Y_N\}$.
 By definition, the MLE is the $\,3 \times 3\,$
parameter matrix $\,\hat \Theta\,$ which maximizes the log-likelihood function.
Thus, we must solve an optimization problem.

\begin{figure}[t]
	\centering
	\includegraphics[width=.43\textwidth]{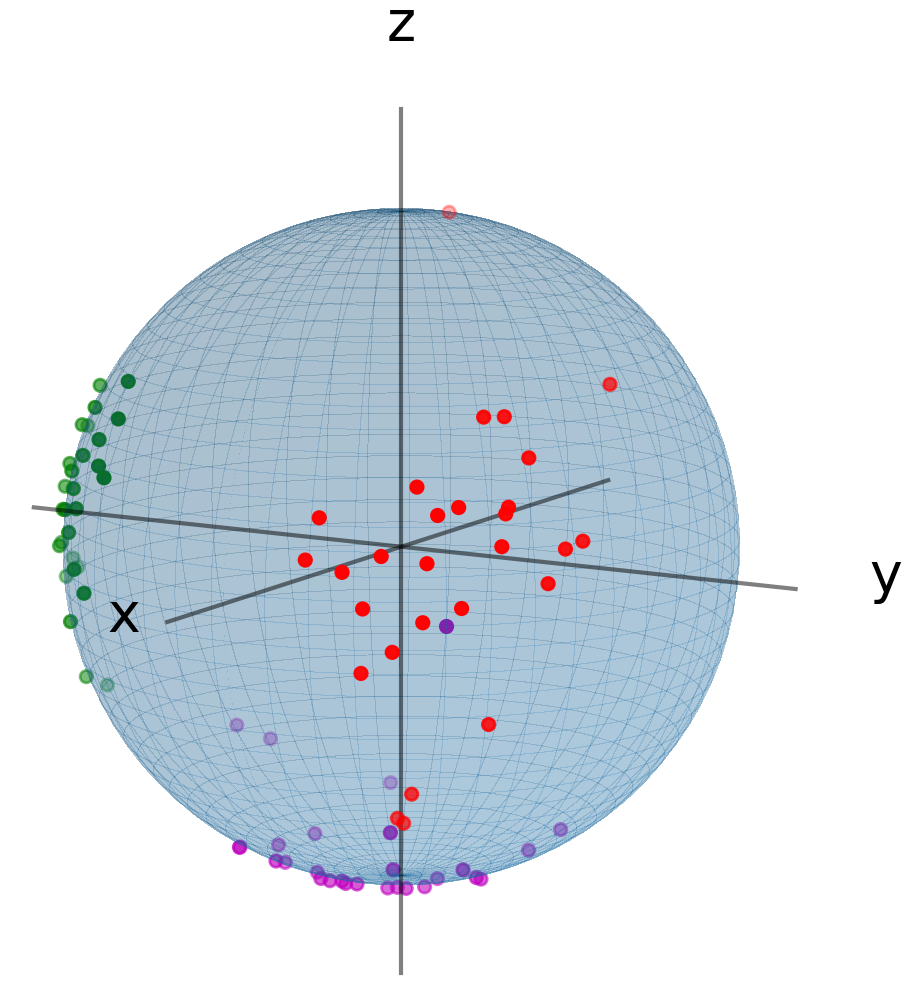}
	\caption{\footnotesize{A dataset of $28$ rotations from a study in
	vectorcardiography \cite{DLM}, a 
	method in medical imaging.
			Each point represents the rotation of the unit standard vector
			on the $x$-axis (depicted in red color),
			the $y$-axis (green), and the $z$-axis (purple). 
			This sample from the group ${\rm SO}(3)$ will be
			analyzed 	in Section \ref{medima}.}}
	\label{fig:MedData}
\end{figure}


From our data we obtain the  {\em sample mean}
$\,\bar{Y}\,=\,\frac{1}{N}\sum_{k=1}^N Y_k $. 
Of course, the sample mean~$\,\bar{Y}\,$ is 
generally not a rotation matrix anymore.
We next compute the sign-preserving singular value decomposition of the sample mean, 
i.e.,~we determine
$Q,R \in {\rm SO}(3)$ such that
$$ \bar{Y} \,\, = \,\, Q \cdot {\rm diag}(g_1,g_2,g_2) \cdot R . $$
The signed singular values $\,g_1,g_2,g_3\,$ together with $Q$ and $R$ are
sufficient statistics for the Fisher model. 
The sample $\,\{Y_1,\ldots,Y_N\}\,$ enters the log-likelihood function only via
 $\,g_1,g_2,g_3$.

\begin{lemma}{\cite[Lemma 2]{sei}} \label{MLEdiag} 
{\em The  log-likelihood function for the given sample from $\,{\rm SO}(3)\,$~is
\begin{equation}\label{MLEmax}
\ell\, \colon \,\RR^3 \longrightarrow \RR,\quad 
x\,\,\mapsto\,\, x_1 g_1 \,+\, x_2 g_2 \,+\, x_3 g_3 \,-\,\log(\tilde c(x_1,x_2,x_3)).
\end{equation}
If $\,(\hat{x}_1,\hat{x}_2,\hat{x}_3)\,$ is the maximizer of the function $\,\ell$, then the matrix
	\mbox{$\,\hat{\Theta}\,=\,Q \diag (\hat{x}_1,\hat{x}_2,\hat{x}_3)R\,$} is
	the MLE of the Fisher model (\ref{eq:fisherdensity})
	 of the sample $\,\{Y_1,\ldots,Y_N\} \,$  from the rotation group~$\,{\rm SO}(3)$. }
\end{lemma}

Lemma~\ref{MLEdiag} says that we need to maximize the function
\eqref{MLEmax} in order to compute the MLE in the Fisher model.
We note that a local maximum is already a global one since 
\eqref{MLEmax} is a strictly concave function.
The maximum is attained at a unique point in $\RR^3$.
We shall compute this point using tools from algebraic analysis
that are discussed in the next section.

\begin{remark} \label{rem:norm}
The singular values of  the sample mean $\,\bar{Y}\,$ are bounded from
above and below, namely $\,1 \geq |g_1| \geq g_2 \geq g_3 \geq 0$.
If $\, g_3\,$ is close to $1$, i.e., the average
of the rotation matrices is almost a rotation matrix, then
the data is typically concentrated about a preferred rotation.
In this case the normalizing constant becomes very large and
MLE on $\,\text{SO}(3)\,$ is numerically intractable;
see also Remark \ref{rem:orbitopes}.
 However, due
to the small spread of the data around a point in $\,\text{SO}(3)$,
a matrix valued Gaussian model on $\,\RR^3\,$ is an accurate
approximation.
\end{remark}

\section{Holonomic representation}\label{sec:holonomic}

We shall represent the normalizing constant 
$\,\tilde c\,$ by a system of linear
differential equations it satisfies. This is known as the holonomic representation
of this function.
We work in the {\em Weyl algebra} $\,D\,$ and in the
 {\em rational Weyl algebra} $\,R\,$ with complex coefficients:
$$ D\,\,=\,\,\mathbb{C}[x_1,x_2,x_3]\langle \partial_1,\partial_2,\partial_3 \rangle 
\qquad {\rm and} \qquad
 R \,\,= \,\, \mathbb{C}(x_1,x_2,x_3)\langle \partial_1,\partial_2,\partial_3\rangle .$$ 
We refer to~\cite{satstu, Tak} for  basics on these two 
noncommutative algebras of linear partial differential operators with 
polynomial and rational function coefficients, respectively. 
In order to stress the number of variables, we sometimes write $\,D_3\,$ instead of $\,D\,$
and $\,R_3\,$ instead of~$R$.
By a {\em $D$-ideal} 
we mean a left ideal in $D$, and by an {\em $R$-ideal} a left ideal in~$R$.
The use of these algebras in statistical inference was pioneered by
Takemura, Takayama, and their collaborators~\cite{HNTT, koyama, holorank, sei, hgmR}. 
We begin with an exposition of their results from~\cite{sei}.

The normalizing constant $\,\tilde c\,$ is closely related to the hypergeometric function 
$\, {_0F_1}\,$ of a matrix argument. In~\cite{sei}, annihilating differential operators of $\,\tilde c\,$ 
are derived from 
\begin{equation} \label{hypermuir}
H_i\,\,=\,\,\partial_i^2 - 1 \,+\, \sum_{j\neq i} \frac{1}{x_i^2-x_j^2}(x_i\partial_i-x_j\partial_j)  
\qquad  {\rm for } \,\,i\,=\,1,2,3.
\end{equation}
These in turn can be
obtained from Muirhead's differential operators in \cite[Theorem 7.5.6]{muir} 
by a change of variables. In the notation of~\cite{satstu}, we have 
$\,H_i \bullet \tilde{c} \,=\,0\,$ for $\,i\,=\,1,2,3$.
Written in the more familiar form of linear PDEs, this~says 
$$
\frac{\partial^2 \tilde c}{\partial x_i ^2} \,+\, \sum_{j\neq i} \frac{1}{x_i^2-x_j^2}
\bigl(x_i\frac{\partial \tilde c}{\partial x_i } 
\,-\, x_j\frac{\partial \tilde c}{\partial x_j} \bigr) \,\,=\,\, \tilde c 
\qquad  {\rm for } \,\,i\,=\,1,2,3. 
$$
Note that the operators $\,H_i\,$ are elements in the rational Weyl algebra~$R$. 
Clearing the denominators, we obtain elements $\,G_i\,$ in the Weyl algebra $\,D\,$ 
that annihilate $\tilde c$,
namely
\begin{equation}\label{hyperpol}
	G_i \,\,=\,\, \prod_{j\neq i} (x_i^2-x_j^2) \cdot H_i.
\end{equation}
By~\cite[Theorem 1]{sei}, 
the following three additional differential operators in $\,D\,$
annihilate $\tilde c$:
\begin{align}\label{so3diff}
 L_{ij} \,\,\coloneqq\,\,(x_i^2-x_j^2)\partial_i\partial_j \,-\, (x_i\partial_i-x_j\partial_j)\,-\,(x_i^2-x_j^2)\partial_{k}.
\end{align}
Here the indices are chosen to satisfy $\,1\leq i < j \leq 3\,$ and $\,\{i,j,k\}\,=\,\{1,2,3\}$. 

Let us consider the $D$-ideal  that is generated by the six operators 
in~\eqref{hyperpol} and~\eqref{so3diff}:
\begin{align}\label{defI}
I \,\,\coloneqq \,\,\langle G_1,G_2,G_3,L_{12},L_{13},L_{23} \rangle.
\end{align} 
In the rational Weyl algebra, we have $\, RI\,=\,\langle H_1,H_2,H_3,L_{12},L_{13},L_{23}\rangle\,$
as $R$-ideals.
We enter the $D$-ideal $\,I\,$ into the
computer algebra system {\tt Singular:Plural} as follows:
\begin{verbatim}
ring r = 0,(x1,x2,x3,d1,d2,d3),dp;
def D = Weyl(r); setring D;
poly L12 = (x1^2-x2^2)*d1*d2 - (x2*d1-x1*d2)-(x1^2-x2^2)*d3;
poly L13 = (x1^2-x3^2)*d1*d3 - (x3*d1-x1*d3)-(x1^2-x3^2)*d2;
poly L23 = (x2^2-x3^2)*d2*d3 - (x3*d2-x2*d3)-(x2^2-x3^2)*d1;
poly G1 = (x1^2-x2^2)*(x1^2-x3^2)*d1^2 + (x1^2-x3^2)*(x1*d1-x2*d2) 
             + (x1^2-x2^2)*(x1*d1-x3*d3) - (x1^2-x2^2)*(x1^2-x3^2);
poly G2 = (x2^2-x1^2)*(x2^2-x3^2)*d2^2 + (x2^2-x3^2)*(x2*d2-x1*d1) 
             + (x2^2-x1^2)*(x2*d2-x3*d3) - (x2^2-x1^2)*(x2^2-x3^2);
poly G3 = (x3^2-x1^2)*(x3^2-x2^2)*d3^2 + (x3^2-x2^2)*(x3*d3-x1*d1) 
             + (x3^2-x1^2)*(x3*d3-x2*d2) - (x3^2-x1^2)*(x3^2-x2^2);
ideal I = L12,L13,L23,G1,G2,G3;
\end{verbatim}
We can now perform various symbolic computations in the Weyl algebra $D$.
We used the libraries {\tt dmodloc}~\cite{dmodloc}  and {\tt dmod}~\cite{LMM}, 
due to Andres, Levandovskyy, and Mart\'in-Morales.
In particular, the following two lines confirm that $\,I\,$ is holonomic
and its holonomic rank is~$4$:
\begin{verbatim}
isHolonomic(I);
holonomicRank(I);
\end{verbatim}

The rank statement means algebraically that
 $\,\dim_{\CC(x_1,x_2,x_3)} \left( R/RI \right) \,=\,4$. In terms of analysis,
 it means that the set of holomorphic solutions to $\,I\,$
 on a small open ball $\,\mathcal{U} \,\subset\, \CC^3\,$ is 
 a $4$-dimensional vector space.
 Here $\,\mathcal{U}\,$ is chosen to be disjoint from the singular locus
 \begin{equation}
 \label{eq:singularlocus}
  {\rm Sing}(I) \,\, = \,\,
 \bigl\{ \,x \in \CC^3 \,:\,
 (x_1^2-x_2^2)(x_1^2-x_3^2)(x_2^2-x_3^2) \,=\, 0 \,\bigr\}. 
 \end{equation}
We note that the normalizing constant $\,\tilde c =
{\tilde c}(x_1,x_2,x_3)\,$ is a real analytic function
on  $\RR^3 \backslash {\rm Sing}(I)$ that extends to a holomorphic function on
all of complex affine space $\CC^3$.
 
Using Gr\"{o}bner bases in the rational Weyl algebra~$R$, we find
that the initial ideal of $\,RI\,$ for the degree reverse lexicographic order
is generated by the symbols of our six operators:
$$ {\rm in}(RI) \,\, = \,\, \langle \,
\partial_1 \partial_2\,,\,
\partial_1 \partial_3\,,\,
\partial_2 \partial_3 \,,\,\,
\partial_1^2\,,\,
\partial_2^2\,,\,
\partial_3^2\,
\rangle. $$
The set of standard monomials equals
$\,S\,=\,\{1,\partial_1,\partial_2,\partial_3\}$. This is a 
$\,\CC(x_1,x_2,x_3)$-basis 
for the vector space $\,R/RI$.
In this situation, we can associate a {\em Pfaffian system} to 
 the $D$-ideal~$\,I$. For the general theory, 
we refer the reader to~\cite{Tak}
and specifically to~\cite[Equation~(23)]{satstu}.

The Pfaffian system is a system of first-order linear
differential equations associated to the holonomic function $\,\tilde c$.
It consists of three $\,4 \times 4\,$ matrices $\,P_1,P_2,P_3\,$
whose entries are rational functions in $\,x_1,x_2,x_3$.
We introduce the column vector 
$\,C\,=\,(\,\tilde c,\,\partial_1\bullet \tilde c,\,\partial_2\bullet \tilde c,\,\partial_3\bullet \tilde c\,)^{\text{t}}$.

\begin{theorem}{\cite[Theorem 2]{sei}}
\label{thm:sei2} \ {\em	The Pfaffian system associated to
the normalizing constant~$\,\tilde c\,$ of the Fisher distribution 
	(\ref{eq:fisherdensity}) consists of the following three vector equations:
		\begin{align}\label{PfaffianSO3}
	\partial_i\bullet C \,\,=\,\,P_i \cdot C \qquad \,\, {\rm for } \,\,i\,=\,1,2,3,
	\end{align}
	where the matrices
	$\,P_1,P_2,P_3\,\in\, \mathbb{C}(x_1,x_2,x_3)^{4\times 4}\,$ are
	 $$
	P_1\,\,=\,\, \begin{pmatrix}
	0 & 1 & 0 & 0\\
	1 & \frac{x_1(-2x_1^2+x_2^2+x_3^2)}{(x_1^2-x_3^2)(x_1^2-x_2^2)} &\frac{x_2}{x_1^2-x_2^2} & \frac{x_3}{x_1^2-x_3^2}\\
	0 & \frac{x_2}{x_1^2-x_2^2} & \frac{-x_1}{x_1^2-x_2^2} & 1\\
	0 & \frac{x_3}{x_1^2-x_3^2}&  1& \frac{-x_1}{x_1^2-x_3^2}
	\end{pmatrix}, \quad
	P_2\,\,=\,\,\begin{pmatrix}
	0 & 0 & 1 & 0\\
	0 & \frac{-x_2}{x_2^2-x_1^2} & \frac{x_1}{x_2^2-x_1^2} & 1\\
	1 & \frac{x_1}{x_2^2-x_1^2} & \frac{x_2(x_1^2-2x_2^2+x_3^2)}{(x_2^2-x_1^2)(x_2^2-x_3^2)} & \frac{x_3}{x_2^2-x_3^2}\\
	0 & 1 & \frac{x_3}{x_2^2-x_3^2} & \frac{-x_2}{x_2^2-x_3^2}
	\end{pmatrix},$$ 
	$$ {\rm and} \quad
	P_3\,\,=\,\, \begin{pmatrix}
	0 & 0 & 0 & 1\\
	0 &  \frac{-x_3}{x_3^2-x_1^2} & 1 &  \frac{x_1}{x_3^2-x_1^2}\\
	0 & 1 &  \frac{-x_3}{x_3^2-x_2^2} &  \frac{x_2}{x_3^2-x_2^2}\\
	1 & \frac{x_1}{x_3^2-x_1^2} &  \frac{x_2}{x_3^2-x_2^2} & \frac{x_3(x_1^2+x_2^2-2x_3^2)}{(x_3^2-x_1^2)(x_3^2-x_2^2)}
	\end{pmatrix}.
	$$
}
\end{theorem}

We reproduced this Pfaffian system from the operators
$\,G_1,G_2,G_3,L_{12},L_{13},L_{23}\,$ with the {\tt Mathematica} 
package {\tt HolonomicFunctions}~\cite{CK}. This was done
by running Gr\"{o}bner basis computations in the rational 
Weyl algebra $\,R\,$ with the degree reverse lexicographic order.
See~\cite[Example 3.4]{satstu} for an illustration on how this is done.

The Pfaffian system~\eqref{PfaffianSO3} allows us to recover 
the $i$th partial derivative of the normalizing constant as the first coordinate of 
the column vector
$P_i\cdot C$. In symbols we have $\,\partial_i \bullet \tilde{c}\,=\,(P_i\cdot C)_1$.
We make extensive use of this fact when 
computing the MLE in Section~\ref{HGDinaction}.
In the same vein, we can recover the Hessian of $\,\tilde{c}\,$
from the Pfaffian system of $\,\tilde{c}\,$ as follows:
\begin{equation}\label{Hessian}
\begin{matrix}
&  \partial_1^2\bullet \tilde{c} \,\, &=\,\, (P_1\cdot C)_2, \quad \partial_1\partial_2 \bullet \tilde{c} \,\,
&=\,\, (P_2\cdot C)_2, \quad \partial_1\partial_3 \bullet \tilde{c} \,\,&=\,\, (P_3\cdot C)_2 ,\\
&\,\partial_2^2\bullet \tilde{c}\,\,&=\,\,(P_2\cdot C)_3, \quad \partial_2\partial_3\bullet \tilde{c} \,\,
&=\,\, (P_3\cdot C)_3, \quad \ \ \, \partial_3^2 \bullet \tilde{c} \,\,&=\,\, (P_3 \cdot C)_4.
\end{matrix}
\end{equation}
This allows for the use of second order optimization algorithms, see Section~\ref{HGDinaction}.

An object of interest---from the algebraic analysis perspective---is
the Weyl closure of the $D$-ideal $\,I$. By definition, the
{\em Weyl closure} is the following $D$-ideal which clearly contains~$I$:
$$W(I)\,\coloneqq \,RI\,\cap\, D.$$
In general, it is a challenging problem to compute the Weyl closure of a $D$-ideal.
This computation is reminiscent of finding the radical of a polynomial ideal, which,
according to Hilbert's Nullstellensatz, consists of all polynomials that vanish on the
complex solutions to the given polynomials. 
The Weyl closure plays a similar role for holonomic functions.
It turns out that computing $\,W(I)\,$ is fairly benign for the $D$-ideal $\,I\,$ studied in this section.

\begin{lemma} \label{lem:W(I)}
Let $\,I\,$ be the holonomic $D$-ideal in~\eqref{defI}. Then  
the Weyl closure $\,W(I)\,$ is generated by $\,I\,$ and the one additional operator 
$\, x_1\partial_1\partial_3+x_2\partial_2\partial_3+x_3\partial_3^2-x_2\partial_1-x_1\partial_2-x_3+2\partial_3 $.
\end{lemma}
\begin{proof}
We used the {\tt Singular} library {\tt dmodloc}~\cite{dmodloc} to compute the Weyl closure of $\,I$. 
We found that $\,I\,$ is not Weyl-closed, i.e., $\,I\,\subsetneq\, W(I)$. 
Moreover, by Gr\"{o}bner basis reductions in the Weyl algebra, we find that 
adding the claimed operator results in a Weyl-closed ideal.
\end{proof}

Following~\cite{koyama, sei}, 
we now consider the Fisher distribution on~$\text{SO}(n)$.
The normalizing constant $\,c(\Theta)\,$ is defined as in~\eqref{eq:normalizing1},
with the integral taken over $\,\text{SO}(n)\,$ with its Haar measure.
Let $\,D_{n^2}\,$ be the Weyl algebra whose variables
are the entries of the $\,n \times n\,$ matrix $\,\Theta \,=\, (t_{ij})$. 
The corresponding $n \times n$ matrix of differential operators in 
$D_{n^2}$ is denoted by $\,\partial \,=\, (\partial_{ij})$.
The following result was established by Koyama~\cite{koyama},
based on earlier work of Sei et al.~\cite{sei}.
We shall prove a more general statement for arbitrary compact Lie groups in Section~\ref{sec6}.

\begin{theorem}\label{ann-theta}
	The annihilator of $\,c(\Theta)\,$ is the $D$-ideal generated by the following operators:
$$
\begin{matrix}	
	\quad d\,\,=\,\,1\,-\,\det(\partial),\quad 
	g_{ij}\,\,=\,\,\delta_{ij}\,-\,\sum_{k=1}^n \partial_{ik}\partial_{jk}\quad  &{\rm for } \,\, 1\le i\le j\le n, \medskip \\
P_{ij}\,\,=\,\,\sum_{k=1}^n \bigl(t_{ik}\partial_{jk}\,-\,t_{jk}\partial_{ik}\bigr) \quad  &{\rm for } \,\, 1\le i<j\le n.
\end{matrix}
$$	
\end{theorem}
Above we omitted half of the equations given in \cite[Equation (12)]{koyama}, 
which is justified by the results in \cite[Section 8.7.3]{procesi}. 
Also, the operators $P_{ij}$ are induced from left matrix multiplication (as in \eqref{eq:lieoperator}) 
rather than right multiplication as in \cite[Equation~(11)]{koyama}.

A problem that was left open in \cite{koyama,sei}, even for $n=3$, 
is the determination of the holonomic rank of $\,J$.
We now address this by introducing dimensionality reduction via invariant theory.
Let $\,J'\,$ be the $D$-ideal generated by the operators $P_{ij}, g_{ij}$. 
This is the analogue of~$\,J\,$
for the orthogonal group ${\rm O}(n)$ in its standard representation in ${\rm GL}_n(\CC)$ (see Section~\ref{sec6}).
Since ${\rm O}(n)$ has two connected components, the corresponding
module in Theorem~\ref{thm:push} is a direct sum of two simple holonomic $D_{n^2}$-modules. 
By symmetry, we obtain 
\begin{equation}\label{eq:oso}
\rank (J') \,\,=\,\, 2 \cdot \rank (J).
\end{equation}
The ring of ${\rm O}(n)$-invariant polynomials on $\,\CC^{n\times n}\,$ is generated by the 
$\,\binom{n+1}{2}\,$ entries $\,\{y_{kl} \}_{1 \leq k \leq l \leq n}$
of the symmetric matrix $\,Y\,=\,\Theta^{\text{t}} \cdot \Theta$ (see \cite[Section 11.2.1]{procesi}).  
These matrix entries
$y_{kl}$ are algebraically independent quadratic forms in
the  $n^2$ unknowns $t_{ij}$.

We now work in the Weyl algebra $\,D_{\binom{n+1}{2}}\,$ with the convention
 $\,y_{kl}\,=\,y_{lk}\,$ and $\,\partial_{kl}\,=\,\partial_{lk}$. Let $\,K\,$ denote
 the left ideal in that Weyl algebra which  is  generated by the operators
\begin{equation}\label{eq:k}
h_{ij} \,\,= \,\, 
2^{\delta_{ij}} \, n \cdot \partial_{ij} \,-\, \delta_{ij}\,+\,
\sum_{k, \, l \,=1}^n 2^{\delta_{ki}+\delta_{lj}} \, y_{kl} \cdot \partial_{ik} \partial_{jl}
\qquad {\rm for} \,\, 1\leq i \leq j \leq n. 
\end{equation}

\begin{theorem}\label{thm:inv}
A holomorphic function is a solution to $\,J'\,$ if and only if it is of the 
form $\,\Theta\, \mapsto \phi(y_{ij}(\Theta))$, where $\,\phi\,$ is a solution to $K$. In particular, 
$\,\rank (K) \, = \, 2 \, \cdot \, \rank (J)$.
\end{theorem}

\begin{proof}
The Lie algebra operators $\,P_{ij}\,$ express left invariance under ${\rm SO}(n)$. 
The fact that every solution to $\,J'\,$ is expressible in $\,Y\,$ follows from 
Luna's Theorem~\cite{luna} (see also~\cite[Section 6.4]{git}). 
We note that the determinant $\,\det(\Theta)\,$ is an $\SO(n)$-invariant that we may omit, 
due to the relation $\,\det(\Theta)^2 \,=\, \det Y$. 
The $D$-ideal $\,K\,$ is the invariant version of~$\,J'$. 
The operator $\,h_{ij}\,$ is derived from $\,g_{ij}\,$ by the chain rule.
The result therefore follows from~\eqref{eq:oso}.
\end{proof}

As an application of Theorem~\ref{thm:inv}, we answer a question left open in~\cite[Proposition 2]{sei}.
\begin{proposition}\label{prop:rank4} 
For $n\,=\,3$, we have~$\rank (J)\,=\,4$.
\end{proposition}

\begin{proof}
We used the computer algebra system {\tt Macaulay2} \cite{M2}. 
Unlike for $\,\rank (J)$, the calculations for $\,\rank (K)\,$ finished, 
and we found $\,\rank (K)\,=\,8$. 
We conclude by Theorem~\ref{thm:inv}.
\end{proof}


We next explain how the $D_{n^2}$-ideal $\,J$ and the 
$D_n$-ideal $\,I\,$ in~\eqref{defI} are connected. 
The ideal~$\,I\,$ is defined as in (\ref{defI}) for all $n$. We 
use the construction of the {\em restriction ideal}. For the general definition see
\cite[Equation (13)]{satstu}. In our case, the construction works as follows.
We set $\,x_i \,=\, t_{ii}\,$ for $\,i\,=\,1,\ldots,n\,$ and we write $\,D_n\,$ for the corresponding 
Weyl algebra. Then
\begin{equation}
\label{eq:diag}
J_{\text{diag}} \,\,\coloneqq\,\,\left( \,J \, + \, 
\bigl\{ \,t_{ij}\,:\, 1 \leq i \not= j \leq n \bigr\} \cdot D_{n^2}\, \right) \,\, \cap \, \,D_n
\end{equation}
is the $D_n$-ideal obtained by restricting the annihilator of $\,c(\Theta)\,$ 
to the diagonal entries of the matrix $\,\Theta$. Note that the second summand 
in~\eqref{eq:diag} is a {\em right} ideal in the Weyl algebra $\,D_{n^2}$.

If $\,f(\Theta)\,$ is a function in the $\,n^2\,$ variables $\,t_{ij}\,$ that is annihilated by $\,J$, 
then the restriction ideal
$\,J_{\rm diag}\,$ annihilates the function $\,f( {\rm diag}(x_1,\ldots,x_n))\,$ in $\,n\,$ variables.
Therefore, $\,J_{\rm diag}\,$ annihilates the restricted
normalizing constant $\tilde c(x_1,\ldots,x_n)$.
We have the following result.

\begin{proposition}
\label{thm:inclusions}
The following inclusions hold among holonomic $D_n$-ideals representing~$\tilde c$:
$$ I \,\,\subseteq \,\, J_{\rm diag} \,\,\subsetneq  \,\,W(J_{\rm diag})
\,\, \subseteq \,\,{\rm ann}_{D_n}(\tilde c).  $$
Equality holds for $n \leq 3$ in the rightmost inclusion.
\end{proposition}

\begin{proof}[Proof sketch]
The proof of~\cite[Theorem 1]{sei} shows that~$\,I\,$
is contained in $\,J_{\diag}$.
The middle inclusion is strict by
Lemma \ref{lem:W(I)}. We have
$\,W(J_{\rm diag})\,\subseteq\,
{\rm ann}_{D_n}(\tilde c)\,$ because
the annihilator of a smooth function such as $\,\tilde c\,$ is Weyl-closed,
by an argument spelled out in~\cite{GLS}.  

The equality on the right for $n=3$ is shown by proving 
$\,W(I) \,=\, {\rm ann}_{D_3}(\tilde c)$.
We use the following  argument and computations. 
The Fourier transform $\,W(I)^{\mathcal F}\,$  is the $D$-ideal 
obtained by switching $\partial_i$ and $x_i$ (up to sign).
We find that its holonomic rank is~$1$.
We next compute the {\em holonomic dual} of the module $\,D_3/W(I)^{\mathcal F}$.
This is another $D_3$-module, as defined in \cite[Section 2.6]{htt}.
There is a built-in command for the holonomic dual
in {\tt Macaulay2}~\cite{M2}.
Another computation, using localization techniques, verifies that both $\,D_3/W(I)^{\mathcal F}\,$  
and its holonomic dual are torsion-free as $\CC[x_1,x_2,x_3]$-modules.
These facts imply that $\,D_3/W(I)^{\mathcal F}\,$ is a simple $D$-module, and hence so is $\,D_3/W(I)$. 
From this we conclude that $\,W(I) \,=\, {\rm ann}_{D_3}(\tilde c)$.
\end{proof}

We conjecture that the inclusion on the right is an equality for
all positive integers $n$.
Using results from Section~\ref{sec6}, we can 
argue that $\,W(J_{\rm diag})^{\mathcal{F}}\,$ is regular holonomic for any~$n$.
It appears that its singular locus is a hyperplane arrangement.
The special combinatorial structure encountered in this arrangement gives
strong evidence for the conjecture above.

\section{Maximum likelihood estimation} \label{HGDinaction}

We now proceed to finding the maximum of the log-likelihood 
function of Lemma \ref{MLEdiag} for given datasets.
Since the objective function~\eqref{MLEmax} is strictly 
concave, a local maximum
is the global maximizer and attained at a unique point 
$\hat{x}=(\hat{x}_1,\hat{x}_2,\hat{x}_3)\in \RR^3$.
In order to compute~$\hat{x}$, we run a number of algorithms,
each using the holonomic gradient method. This is based on
the results presented in the previous section, especially
on Theorem \ref{thm:sei2} and Equation~\eqref{Hessian}.
These are used to compute the function values, gradients, and Hessians
in each iteration.

A critical step in running any local optimization method is finding a suitable starting point.
As mentioned in Section \ref{sec:holonomic}, solutions to the $D$-ideal $\,I\,$ are 
analytic outside the singular locus $\text{Sing}(I)$. Starting points need to be chosen 
in~$\RR^3\backslash \text{Sing}(I)$.
For the Fisher model on $\text{SO}(3)$, the singular locus $\text{Sing}(I)$ is 
the arrangement (\ref{eq:singularlocus}) of six planes through the origin in~$\RR^3$.
This partitions~$\RR^3$ into $24$ distinct chambers. For the algorithms  
described below,
we choose starting points in each of the $24$ connected components of 
$\,\RR^3/\text{Sing}(I)$,  and we evaluate
the vector $C$ at these points. This initialization can be done either via the series expansion 
method of \cite[Section 3.2]{sei} or using the package {\tt hgm} \cite{hgmR} 
in the statistical software~{\tt R}.

In this section, we present three optimization methods based on algebraic analysis.
The simplest is {\em Holonomic Gradient Ascent} (HGA). This
is a straightforward adaptation of the HGD method in~\cite{sei}.
Second, we introduce a holonomic version
 of the  Broyden--Fletcher--Goldfarb--Shanno (BFGS) method \cite[Chapter 6, \S 1]{NocWri}.
 BFGS is a quasi-Newton method that requires the
gradient and the function value as inputs. Both can be calculated directly
using (\ref{PfaffianSO3}). This turns BFGS into {\em Holonomic BFGS} (H-BFGS).
The third algorithm to be introduced is a {\em Holonomic Newton Method}.
This second-order method exploits the fact that
the Hessian is easy to calculate from (\ref{Hessian})
and that the objective function is strictly concave.

To get started, we need 
an expression for the gradient of the log-likelihood function $\,\ell\,$ and a
Holonomic Gradient Method (HGM) for evaluating that expression.
By Lemma~\ref{MLEdiag}, 
\begin{equation} \label{eq:nablaell}
\nabla\ell(x) \,\,= \,\,\begin{pmatrix}
g_1\\ g_2\\ g_3
\end{pmatrix} \,-\, \frac{1}{\tilde c(x)} \cdot \nabla \tilde c(x).
\end{equation}
Note that 
$\,C(x) \,=\, (\tilde c(x),\nabla \tilde c(x))^\text{t}$.
Hence, our task to evaluate $\,\nabla\ell\,$ at any point 
amounts to evaluating the vector-valued function $\,C\,$ at any point. 
This is where the HGM comes in.

In general, we approximate the function 
$\,C\,$ at a point $\,x^{(n+1)}\,$ given its value at a previous point
$\,x^{(n)}\,$.
To this end, a path 
$\,x^{(n)}\to x^{(n)}+\delta^{(1)}\to x^{(n)}+\delta^{(2)}\to\dots\to x^{(n)}+\delta^{(K)}\to x^{(n+1)}\,$
is chosen, where $\,\delta^{(1)},\ldots,\delta^{(K)}\,\in\, \RR^3\,$ with 
$\,\Vert \delta^{(m+1)}-\delta^{(m)} \Vert\,$ sufficiently small.
The linear part of the Taylor series expansion
of $C$ at $x^{(n)}$ yields the following approximations:
\begin{align}
\begin{split}
C(x^{(n)}+\delta^{(m+1)})
&\,\,\approx\,\, C(x^{(n)}+\delta^{(m)}) \,+\, \sum_{i=1}^3 \,(\delta^{(m+1)}_i - \delta^{(m)}_i)\,(\partial_i\bullet C)(x^{(n)}+\delta^{(m)})\\
&\,\,=\,\, C(x^{(n)}+\delta^{(m)}) \,+\, \sum_{i=1}^3 \,(\delta^{(m+1)}_i - \delta^{(m)}_i) \, P_i  \cdot C(x^{(n)}+\delta^{(m)}).
\end{split}
\label{eq:hgm}
\end{align}
We choose a path consisting of points, separated by intervals of size~$\Delta t$, 
on the line segment 
$\,x(t) \,=\, x^{(n)}(1-t) \,+\, x^{(n+1)}t\,$ with $t\in[0,1]$. 
With this notation, Equation~\eqref{eq:hgm} becomes
\begin{equation}
\begin{matrix}
C(x((m+1)\Delta t)) \,\,\approx \,\,
C(x(m\Delta t)) \,+\, \sum_{i=1}^3\, (x^{(n+1)}_i-x^{(n)}_i)\,\Delta t \cdot \,P_i\cdot C(x(m\Delta t)).
\end{matrix}
\end{equation}
If we take the limit $\Delta t \to 0$, then the equation above becomes the differential equation 
$$\frac{dC(t)}{dt} \,\,=\,\, \sum_{i=1}^3\frac{\partial x_i}{\partial t}\frac{\partial C}{\partial x_i} 
\,\,=\,\, \sum_{i=1}^3  \left(x_i^{(n+1)}- x_i^{(n)}\right)P_i \cdot C.$$ 
This ordinary differential equation can be solved using any numerical 
ODE solver, e.g.,~an Euler scheme or Runge--Kutta scheme.
This leads to the following algorithm.

\vspace{12pt}
\begin{algorithm}[H]
	\DontPrintSemicolon
	\KwIn{$x^{(n)}$, $\,x^{(n+1)}$, $\,C(x^{(n)})$, a Pfaffian system $\,P_1,P_2,P_3$}
	\KwOut{$C(x^{(n+1)})$}
	Set $\,x(t) \,=\, x^{(n)}(1-t) \,+\, x^{(n+1)}t$.\;
	Let $\,\frac{dC(t)}{dt} \,=\, \sum_{i=1}^3\frac{\partial x_i}{\partial t}\frac{\partial C}{\partial x_i} \,=\, \sum_{i=1}^3  \left(x_i^{(n+1)}- x_i^{(n)}\right)P_i \cdot C$.\; \label{line:eq}
	Numerically integrate line \ref{line:eq} from $\,t\,=\,0\,$ to $\,t\,=\,1$.
	\caption{Holonomic Gradient Method}
	\label{alg:HGM}
\end{algorithm}
\vspace{12pt}

We employ Algorithm \ref{alg:HGM} as a subroutine for 
the holonomic gradient ascent algorithm, which will be described next.
HGA is analogous to other gradient ascent/descent methods, however, 
with the special feature that 
the gradients are calculated via the HGM algorithm.
A description of the algorithm, adapted for data from $\text{SO}(3)$, is outlined below.

\vspace{12pt}
\begin{algorithm}[H]
	\DontPrintSemicolon
	\KwIn{Matrices $\,Q\,$ and $R$, singular values $\,g_1,g_2,g_3\,$ and a starting point~$x^{(0)} \in\RR^3$}
	\KwResult{A maximum likelihood estimate for the data in the Fisher model 
	(\ref{eq:fisherdensity})}
		Choose a learning rate $\gamma_n$.\;
	Choose a threshold $\delta$.\;
	Evaluate $\,C\,$ at the starting point $x^{(0)}$.\; 
	Evaluate $\,\nabla\ell\,$ at the starting point $x^{(0)}$.\;
	Set $n \,=\,0$.\;
	\While{$\max |\nabla\ell(x^{(n)})|\,<\, \delta$}{
		$x^{(n+1)} \,=\, x^{(n)} \,+\, \gamma_n\nabla\ell(x^{(n)})$.\;
		Calculate $\,C(x^{(n+1)})\,$ via HGM using Algorithm \ref{alg:HGM}.\;
		Calculate $\,\nabla\ell(x^{(n+1)})\,$ from $C(x^{(n+1)})$.\;
		Set $n \,=\,n+1$.\;
	}
	Output vector $\,x^{(n)} \,\in\, \RR^3\,$ as our approximation for $(\hat x_1, \hat x_2, \hat x_3)$. \;
Output the rotation matrix $\,\hat \Theta \, = \, Q \cdot x^{(n)} \cdot R\,$ 
as our approximation for the MLE.
	\caption{Holonomic Gradient Ascent}
	\label{alg:HGA}
\end{algorithm}
\vspace{12pt}

The given data is a list of rotation matrices $\,Y_1,\ldots,Y_N\,$ in ${\rm SO}(3)$.
As explained in Section~\ref{sec2}, we encode these 
in the singular values $\,g_1,g_2,g_3\,$ of the 
sample mean~$\,\bar{Y}\,=\,\frac{1}{N}\sum_{k=1}^N Y_k $.
Thus, the input for HGA consists primarily of
just three numbers $\,g_1,g_2,g_3$. 
They are used in the evaluation in the first terms of~$\nabla \ell$, 
as seen in (\ref{eq:nablaell}). The second term is evaluated 
by matrix multiplication with $\,P_1,P_2,P_3$, as seen in~(\ref{PfaffianSO3}).
Part of the input are also the matrices $\,Q\,$ and $\,R\,$ that 
 diagonalize the sample mean
$\,\bar{Y}$. They are needed in the last step to recover $\,\hat \Theta\,$ from
 $\,\hat x_1,\hat x_2, \hat x_3 \,$ as in Lemma~\ref{MLEdiag}.	
The HGA algorithm has two parameters, namely the threshold $\,\delta\,$ 
which indicates a termination condition,
and the learning rate $\gamma_n$. 
While $\,\delta\,$ can be chosen freely depending on the desired accuracy,
choosing the learning rate can have significant effects on the convergence of the algorithm. 
In our computations we chose  $\,\gamma_n \,=\, 10^{-2}$.
This can clearly be improved. However, the standard technique of
performing line searches to find a good $\,\gamma_n\,$ is not recommended 
as evaluating $\,C\,$ at a new point is costly. 

To employ more advanced methods such as BFGS, and  to avoid 
integrating along a path crossing the singular locus, we use~\cite[Corollary 1]{sei}.
This  states that the value of $\,C\,$ at a point $(x_1,x_2,x_3)$ can be obtained
by integrating the following  ODE from $t = \epsilon \ll 1\,$ to~$t=1$:
\begin{equation}
\frac{dC}{dt} \,\,=\,\, \begin{pmatrix}
0 & x_1 & x_2 & x_3\\
x_1 & -2/t & x_3 & x_2\\
x_2 & x_3 & -2/t & x_1\\
x_3 & x_2 & x_1 & -2/t
\end{pmatrix}\cdot C.
\label{eq:CorC}
\end{equation}

Using this approach for calculating $\,C$, we can employ BFGS optimization using HGM 
as a subroutine to calculate the gradients and function values required as inputs.
The H-BFGS method achieves much faster convergence rates than the simple HGA algorithm~\ref{alg:HGA}.

A final very powerful algorithm for concave (or convex) functions is the 
Newton method which uses the Hessian matrix.
Often, finding the Hessian matrix ${\,\bf H}[\ell(x)]\,$ of a function is a difficult task. 
However, using holonomic methods
the Hessian is a obtained for free via
$$
\partial_i\partial_j\bullet \ell \,\,\,=\,\,\, \frac{1}{\tilde{c}^2}\,(\partial_i\bullet \tilde{c})\,
(\partial_j\bullet\tilde{c}) \,-\, \frac{1}{\tilde{c}}\,\partial_i\partial_j\bullet \tilde{c},
$$
and the relations in (\ref{PfaffianSO3}) and (\ref{Hessian}).
We found that the Newton method,
$$
x^{(n+1)} \,\,=\,\, x^{(n)} \,-\, {\bf H}[\ell (x)]^{-1}\cdot \nabla \ell (x),
$$
gives the fastest convergence. We refer to this approach as the {\em Holonomic Newton Method}.

 We implemented the H-BFGS method  in a script in the software {\tt R}.
 Interested readers may obtain our implementation from the first author.
This code is custom-tailored for rotations in $3$-space.
The function $C$ is evaluated at the starting point $\,x^{(0)}\,$ using the series expansion method 
that is described in~\cite[Section 3.2]{sei}. Here we truncate the series at order~$41$.

\begin{example} \label{ourdata}
We created a synthetic dataset consisting of $\,N=500\,$ rotation matrices.
These were sampled from the Fisher distribution with parameter matrix
\begin{equation}
\Theta \,\,=\,\, \begin{pmatrix}
-1.178 & 0.2804 & 1.037\\
-0.3825 & 0.9181 & 0.6016\\
-0.0955 & 0.9037 & 1.695
\end{pmatrix}.
\end{equation}
The sample mean and its sign-preserving singular value decomposition are
found to be
\begin{align*}
&\bar{Y} \,\,=\,\, \begin{pmatrix}
-0.2262 & 0.1021 & 0.2260\\
-0.0233 & 0.0611 & 0.2779\\
-0.0364 & 0.2802 & 0.3529
\end{pmatrix}\,\, = \,\, Q\cdot\begin{pmatrix}
0.5946 & 0.0000 & 0.0000\\
0.0000 & 0.1838 & 0.0000\\
0.0000 & 0.0000 & 0.1059
\end{pmatrix}\cdot R, \\ {\rm with} \quad
& Q \,\,=\,\, \begin{pmatrix}
-0.4977 & 0.8589 & 0.1211\\
-0.4518 & -0.1376 & -0.8815\\
-0.7404 & -0.4934 & 0.4565
\end{pmatrix},\quad R \,\,=\,\, \begin{pmatrix}
0.2524 & -0.4808 & -0.8397\\
-0.9419 & -0.3209 & -0.0993\\
-0.2217 & 0.8160 & -0.5339
\end{pmatrix}.
\end{align*}
Running \mbox{H-BFGS}  on this input, the MLE is found to be
\begin{equation}
\hat{\Theta} \,\,=\,\, \begin{pmatrix}
-0.8972 & 0.3446 & 0.9682\\
-0.2392 & 0.7777 & 0.7856\\
-0.0763 & 0.8664 & 1.616
\end{pmatrix}\, \,=\,\, Q\cdot \begin{pmatrix}
2.422 & 0.0000 & 0.0000\\
0.0000 & 0.7432 & 0.0000\\
0.0000 & 0.0000 & -0.3043
\end{pmatrix}\cdot R.
\end{equation}
While the entries of the MLE $\, \hat{\Theta}\,$ have the correct sign and order of magnitude, 
the actual values are not very close to those in~$\,\Theta$.
In order to isolate the effect of the sample size on the MLE, we extended the data to $10 000$ matrices. 
In the iterations we recorded
the Frobenius distance (FD) from~$\, \hat{\Theta}\,$ to~$\,\Theta\,$ and the logarithm of the 
likelihood ratio (LR) of the exact parameter and the MLE. 
Our findings are outlined in the table below.
\begin{center}
\begin{tabular}{|r|r|r|r|r|}
	\hline
	\textbf{\# Data} & \textbf{H-BFGS FD} & \textbf{Newton FD} &\textbf{ \mbox{H-BFGS} LR} & \textbf{Newton LR}\\
	\hline
	1000 & 0.2136 & 0.2136 & 0.07006 & -0.0007\\ 
	2000 & 0.1145 & 0.1145 & 0.08828 & -0.0013\\
	3000 & 0.1155 & 0.1155 & 0.07837 & -0.0012\\
	4000 & 0.1485 & 0.1485 & 0.08185 & -0.0014\\
	5000 & 0.1700 & 0.1700 & 0.07439 & -0.0009\\
	6000 & 0.1247 & 0.1247 & 0.07325 & -0.0006\\
	7000 & 0.1321 & 0.1321 & 0.07248 & -0.0006\\
	8000 & 0.1011 & 0.1011 & 0.07294 & -0.0003\\
	9000 & 0.0985 & 0.0985 & 0.07127 & -0.0002\\
	10000 & 0.0838 & 0.0838 & 0.07219 & -0.0002\\
	\hline
\end{tabular}
\end{center}

\medskip

In our experiments we found that
the  convergence in likelihood ratio and Frobenius distance is slow.
It appears that, in general, the MLE problem is not very well conditioned.
\end{example}

\begin{remark}
	The authors in~\cite{sei} report that the HGD algorithm becomes numerically unstable 
when it is	close to the singular locus of the Pfaffian system.
	They recommend~picking a starting point in the same connected component
	of $\,\RR^3 \setminus {\rm Sing}(I)\,$
	where the MLE is suspected.
	In contrast, our computations suggest that the output of the HGA does not depend 
	on the connected component which the starting point lies in, 
	when a sufficiently stable numerical integration method (e.g.~{\,\tt lsode\,} 
	from the {\,\tt R\,} package {\,\tt deSolve}\,) is chosen in Algorithm~\ref{alg:HGM}.
	\end{remark}
	
\begin{remark}	\label{rem:orbitopes}
The sample mean matrix $\,\bar{Y}\,$ lies in the convex
hull of the rotation group. This convex body, denoted ${\rm conv}({\rm SO}(3))$,
was studied in~\cite[Section 4.4]{orbi}, and an explicit
representation as a spectrahedron was given in~\cite[Proposition 4.1]{orbi}.
It follows from the theory of orbitopes~\cite{orbi} that
the singular values of matrices in $\,{\rm conv}({\rm SO}(3))\,$
are precisely the triples that satisfy $\,1 \geq |g_1| \geq g_2 \geq g_3 \geq 0$. 
These inequalities define two polytopes, which are responsible for
the facial description of $\,{\rm conv}({\rm SO}(3))\,$ found in~\cite[Theorem~4.11]{orbi}.

We can think of the MLE as a map from the
interior of the orbitope $\,{\rm conv}({\rm SO}(3))\,$
to~$\RR^3$. Using the singular value decomposition, we restricted this map
 to the open polytopes given by $1 > |g_1| > g_2 > g_3 > 0$.
Note that the coordinates of the vector $\,\hat{x}\,$ goes off to infinity as 
the maximum of $\, \{ g_1,g_2,g_3\}\,$ approaches~$\,1$.
This follows from~\cite[Equation~(4.12)]{KM},
where the analogue for $\,\text{O}(n)\,$ was derived.
This divergence can cause numerical problems.
\end{remark}

In this section, we have turned the earlier results on $D$-ideals into 
numerical algorithms. This is just a first step.
The success of any local method relies heavily
on a clear understanding of the numerical analysis
that is relevant for the problem at hand.
A future study of condition numbers from the perspective
of holonomic representations would be desirable.

\section{Rotation data in the sciences} \label{sec5}

Rotation data arise in any field of science in which the orientation of an object
in $3$-space is important. Occurrences include
a diverse number of research areas such as
medical imaging, biomechanics, astronomy, geology, and materials science. 
In this section, we apply our methods to a 
prominent dataset of vectorcardiograms and to biomechanical data. 
We also review previous findings
on rotation data in astronomy, geology, and materials science.

\subsection{Medical imaging}
\label{medima}

One important occurrence of rotational data in the applied sciences 
stems from medical imaging, 
and more precisely from vectorcardiography.
In that field, the electrical forces generated by the heart~are studied
 and their magnitude and direction are recorded. 
 
The dataset presented in~\cite{DLM} is a famous example of 
directional data. It contains the orientation of the 
vectorcardiogram (VC) loop of $98$ children aged~$2-19$. 
In particular, the orientation is measured using two 
different techniques. Both measurements are given in the 
form of two vectors. The first identifies the VC loop of greatest magnitude and 
the second is  the normal direction to the~loop. 
We add as a third vector the cross product of the magnitude 
and normal vector to form a right handed set and, therefore, a 
rotation matrix.

This dataset has 
been used to exemplify a range of methods in directional statistics, 
see, e.g.,~\cite{PrenII}. We applied the optimization methods from
Section \ref{HGDinaction} to the same dataset. In other words,
we computed the maximum of the log-likelihood function
\eqref{MLEmax} for the orientations of the VC loop.
In order to match our analysis with the results of
\cite{PrenII}, we only consider the $28$ data points of the boys aged $2-10$.
A colorful illustration of the action of these $28$ rotation matrices on
the coordinate axes is shown in Figure~\ref{fig:MedData}.

 We now proceed
to the MLE. The sample mean has the  singular valued decomposition
\begin{equation}
\bar{Y} \,\,=\,\, \begin{pmatrix}
0.6868 & 0.5756 & 0.1828\\
0.5511 & -0.7372 & -0.0045\\
0.1216 & 0.1417 & -0.8630
\end{pmatrix} \,\,=\,\, Q \cdot\begin{pmatrix}
0.9469 & 0.0000 & 0.0000\\
0.0000 & 0.8962 & 0.0000\\
0.0000 & 0.0000 & 0.8737
\end{pmatrix}\cdot R,
\end{equation}
where
\begin{equation}
Q \,\,=\,\, \begin{pmatrix}
0.6112 & 0.7636 & 0.2079\\
-0.7498 & 0.4748 & 0.4608\\
0.2532 & -0.4376 & 0.8628
\end{pmatrix},\quad R \,\,=\,\, \begin{pmatrix}
0.03941 & 0.99324 & -0.1092\\
0.81778 & 0.03072 & 0.5747\\
0.57418 & -0.11194 & -0.8110
\end{pmatrix}.
\end{equation}
By forming the matrix product $\,QR\,$ we recover the
result of~\cite{PrenII}. The matrix $QR$, however, is only one part
of the MLE as described in~\cite{KM}.
By using \mbox{H-BFGS}, we can find the full
MLE of the Fisher model. We compared \mbox{H-BFGS}  to other methods. 
For that, we estimated
$\,x_1,x_2,x_3\,$ with a BFGS optimization of the log-likelihood using the
series expansion of the normalizing constant. We then compare the
resulting estimate to the output of \mbox{H-BFGS}.  

The \mbox{H-BFGS} algorithm finds the MLE 
$$\hat{x}_1 \,\,=\,\, 20.072407,\quad \hat{x}_2  \,\,=\,\,  12.513841, \quad  \hat{x}_3  \,\,=\,\,  -6.510704,$$ 
which corresponds to a log-likelihood of $\,\hat{\ell}  \,=\,  3.97299$. 
The runtime of the algorithm is highly dependent on the number of non-zero terms 
in the series expansion for $\tilde{c}$. 
In this calculation, the first $6000$ non-zero terms are used and the runtime is about $4$ seconds.
The classical BFGS method is not convergent if only the first $6000$ non-zero terms are used. 
Hence, we need to truncate the series expansion at higher order. 
If we use the first $48000$ non-zero terms, 
then the series expansion BFGS method finds the MLE 
$\,\hat{x}_1  \,=\,  17.604156,\, \hat{x}_2  \,=\,  10.024591,\, \hat{x}_3  \,=\,  -3.881811$,
which gives~$\hat{\ell}  \,=\,  3.96330$. The computation takes about $20$ seconds.
Hence, the holonomic BFGS outperformed the classical method by finding a better 
likelihood value in much shorter time.

\subsection{Biomechanics}

Rotational data is ubiquitous in the biomedical sciences. 
A prominent experiment in this area is the human kinematics 
study of~\cite{RRA}. In this experiment, the rotations of four 
different upper body parts were tracked while the subject 
was drilling holes into six different locations of 
a vertical panel. In~\cite{BNVII}, this dataset was studied and 
maximum likelihood and Bayesian point estimates for the 
orientation of the wrist were obtained and credible regions 
constructed.

A further experiment concerns the heel orientation of primates. 
In the experiments, the rotation of the calcaneus bone (the heel) 
and the cuboid bone, which is horizontally adjacent to the 
heel and closer to the toes, was measured. A load was applied 
to three sedentary primates, a human, a chimpanzee, and a 
baboon and the rotation of their ankle was recorded. While 
the data is actually a time series, the simplifying assumption 
of independent identically distributed data is made 
in its analysis~\cite{BNV}. 
We study this dataset which was kindly provided by 
Melissa Bingham. 
The sample mean for the human data equals

\begin{equation} \label{eq:YQR}
\bar{Y} \,\,=\,\, \begin{pmatrix}
	-0.1013 & -0.9127 & -0.3811\\
	0.3275 & -0.3895 & 0.8535\\
	-0.9335 & -0.0358 & 0.3475
\end{pmatrix} \,\,=\,\, Q\cdot \begin{pmatrix}
	0.9997 & 0.0000 & 0.0000\\
	0.0000 & 0.9926 & 0.0000\\
	0.0000 & 0.0000 & 0.9923
\end{pmatrix}\cdot R,
\end{equation}
with
$$
Q \,\,=\,\, \begin{pmatrix}
	0.4771 & 0.8753 & -0.0791\\
	-0.4320 & 0.1552 & -0.8884\\
	-0.7654 &  0.4580 & 0.4521
\end{pmatrix}\,\,\text{\normalsize and } \,\,R \,\,=\,\, \begin{pmatrix}
	0.5248 & -0.2399 & -0.8167\\
	-0.4690 & -0.8822 & -0.0422\\
	-0.7104 & 0.4051  & -0.5754
\end{pmatrix}.
$$

We see on the right hand side in~\eqref{eq:YQR} that the singular values 
for this dataset only differ in the third significant figure and the smallest 
singular value is approximately $1$. We found that the 
normalizing constant gets too large to be computed directly.
Indeed,
our simulations returned a value error when $\,\tilde{c} \approx 10^{308}$.
This is a serious numerical issue, arising in any MLE algorithm that attempts to directly 
calculate $\,\tilde{c}\,$ when the sample mean is almost a rotation
matrix. Singular values close to one imply that the samples are 
 concentrated on the unit sphere. One could either 
use a rotational Maxwell distribution~\cite{Kag} as a local model 
or the approximation used in~\cite{BNV}. The data for the 
baboon and the chimpanzee show similar traits.

We found that  progress can be made by applying a gauge transform in Equation~\eqref{eq:CorC}, aimed at
scaling the input for H-BFGS.
Let $\,\lambda_0\,$ be the largest eigenvalue of 
$$
A \,\,=\,\, \begin{pmatrix}
0 & x_1 & x_2 & x_3\\
x_1 & 0 & x_3 & x_2\\
x_2 & x_3 & 0 & x_1\\
x_3 & x_2 & x_1 & 0
\end{pmatrix}.
$$
We can derive an ODE for the function $\,D \,=\, C\cdot \text{exp}(-\lambda_0 t)\,$ 
from Equation~\eqref{eq:CorC}. 
The function $\,D\,$ is guaranteed to have smaller values than $\,C$. Furthermore,
the ratio $\,(\partial_i \bullet\tilde{c})/\tilde{c} \,=\, C_i/C_0 \,=\, D_i/D_0\,$ is invariant. 
Despite being able to compute $\,\log(\tilde{c})\,$ using the gauge transformation, MLE
becomes very unstable due to the numerical accuracy required. Finding the MLE from a random
starting point using H-BFGS proved intractable. 
However, using the asymptotic formula of~\cite{KM} to provide a suitable starting point for H-BFGS, 
we found the MLE $\hat{x}_1 = 5543.106,\, \hat{x}_2 = 3753.078,\, \hat{x}_3 = -3685.242$ 
corresponding to a log-likelihood of $\hat{\ell} = 10.59342$.
The asymptotic formula yielded an MLE of 
$\hat{x}_1 = 5543.102,\, \hat{x}_2 = 3753.025,\, \hat{x}_3 = -3685.298$ and $\hat{\ell} = 10.52366$. 
Hence, H-BFGS finds a slightly better MLE than the asymptotic formula.

\subsection{Astronomy and geology}

Astronomical applications of the matrix Fisher model on $\,\text{SO}(3)\,$ 
are often concerned with the orbits of near earth objects~\cite{MJ,sei}. 
Such objects are comets or asteroids in an elliptic orbit around the sun 
with the sun in their focus.
The data comes as sets of vectors in~$\,\RR^3\,$ 
taking the sun as the origin. The first vector, $X_1$, is 
the perihelion direction, which points to the
location on the orbit closest to the sun. The second vector, 
$X_2$, is the unit normal to the orbit. Together with their 
cross product these vectors form a right 
handed set.
 Therefore, they define a rotation matrix. 
Questions of astronomical interest are whether the 
perihelion direction is uniformly distributed on the 
sphere and whether the orbit orientations are uniform 
on~$\,\text{SO}(3)$. To answer the latter question the Raleigh 
statistic can be used~\cite{MJ,sei}.

Sei et al.~\cite{sei} studied a dataset of rotations representing 
$151$ comets and $6496$ asteroids.
They computed maximum likelihood estimates using the holonomic gradient method and 
also series expansions. The Raleigh statistic for the 
dataset was calculated and the null hypothesis of a uniform 
distribution was strongly rejected. Further, the hypothesis 
of the data originating from a Fisher distribution on a Stiefel 
manifold was tested against the hypothesis of $\,\text{SO}(3)$, and 
the evidence strongly suggested to reject the Stiefel manifold.

Rotations arise in geology and earth sciences in the 
study of earthquake epicenters~\cite{Kag} and the analysis of
plate tectonics~\cite{DT}. 
Davis and Titus~\cite{DT} studied a dataset of the deformation 
of a shear zone in northern Idaho. However, this was done
in the context of invalidating a geology inspired model that 
had been used previously to explain the shear deformations. 

Kagan \cite{Kag} studied rotational data describing the earthquake focal 
mechanism orientation.  Various models, including 
the Fisher model, were discussed in this article. However, the Fisher model 
was dismissed due to the difficulty of normalization for 
small spread data as discussed in Remark~\ref{rem:norm}. 
The alternative model used in \cite{Kag} was a rotational Maxwell distribution 
as a local approximation. Our results offer a chance to revisit the Fisher model.

\subsection{Materials science}

One important source of rotational data is materials science, 
where patterns from electron 
backscatter diffraction (EBSD) 
are analyzed (see, e.g.,~\cite{BHJPSW}). 
This type of data provides information about the orientation 
of grains within a material. Crystal orientation has important implications 
on the  properties of polycrystalline materials. 
One issue with EBSD data is the fact that orientations of the crystals can only be 
determined within a coset of the crystallographic group the 
grain belongs to. This is due to the fact that a crystal is a 
lattice and every lattice comes with certain translational and 
rotational symmetries. Orientations can only 
be determined up to the rotational invariance of the lattice.
Hence, the data, although giving information about rotations, is
strictly speaking not on~$\SO(3)$, but on its quotient by
a discrete symmetry subgroup. To adapt our analysis, an appropriate
parametrization or embedding for such a quotient needs to be found.
This, however, is beyond the scope of this paper and is left for
future work. Before going to such manifolds, we start with Lie groups.

\section{Compact Lie groups}\label{sec6}

The Fisher model on $\,\SO(n)\,$ generalizes naturally to other compact Lie groups.
We define the Fisher distribution and the normalizing constant as in \eqref{eq:fisherdensity} 
and \eqref{eq:normalizing1}, but with integration over 
the Haar measure on the Lie group. In this section, we
introduce these objects and their holonomic representation.
In particular, we establish the analogue of
Theorem \ref{ann-theta} for compact Lie groups.
This opens up the possibility of applying algebraic analysis 
to data sampled from manifolds other than $\,\SO(n)\,$ provided these have
the structure of a group.

Let $\,G\,$ be a compact connected Lie group and fix a 
real representation \mbox{$\,\pi: G\,\to\, \GL_n(\RR)$}. We can assume that $\,\pi\,$ is injective,
i.e.,~the representation is faithful.
We note that any compact Lie group admits a faithful representation~\cite[Section 8.3.4]{procesi}. 
The matrix group
 $\,\pi(G) \,\subset\, \RR^{n\times n}\,$ is a closed algebraic subvariety (see \cite[Section 8.7]{procesi}). 
If one starts with a complex representation instead, the situation can be studied in
the polynomial ring over $\CC$.

For our algebraic approach, the ambient setting is
the complex affine space \mbox{$X\,\coloneqq \,\CC^{n\times n}$}. 
The complexification $\,G_\CC\,$  of our group $\,G\,$ is a complex connected reductive 
algebraic~group \cite[Section 8.7.2]{procesi}. 
The extension $\,\pi: G_\CC \to X\,$ is a closed embedding.
Its image,  the matrix group $\pi(G_\CC)$, is the complex affine variety in $\,X$,
cut out by the same polynomials as the ones defining $\pi(G)$. 
We denote by $\,I_G\,$ the ideal generated by these polynomials in~$\CC[X]$. 
The quotient ring $\,\CC[G] \,\coloneqq\, \CC[X]/I_G\,$ is the
ring of polynomial functions on the group $\pi(G_\CC)$.

Let $\,\lie\,$ denote the complex Lie algebra of $G_\CC$. This is the complexification of the
real Lie algebra of the given Lie group $G$.  We write $\,U(\lie)\,$ for the universal 
enveloping algebra of~$\lie$.
For any affine variety, one can define the ring of algebraic differential operators
on that variety.
This is generally a complicated object, but things are quite nice in our case.

Let $\,D_G\,$ denote the ring of differential operators on $G_\CC$. We have natural inclusions
 $$ \lie \,\subset\, U(\lie) \,\subset\, D_G \qquad {\rm and} \qquad \CC[G] \,\subset\, D_G. $$
 These inclusions exhibit desirable properties.
   Namely, we have canonical
 isomorphisms
\begin{equation}\label{eq:tantriv}
D_G \,\, \cong \,\, \CC[G] \, \otimes \,U(\lie)\,\, \cong \,\, U(\lie)\, \otimes \,\CC[G].
\end{equation}
This holds because
 left (or right) invariant vector fields of $\,G_\CC\,$ trivialize the tangent bundle.
Recall that $\,G_\CC\,$ acts on $\,X\,=\,\CC^{n \times n}\,$ by left matrix multiplication via $\pi$. 
Through this action, elements in the Lie algebra $\,\mathfrak{g}\,$ induce vector fields on~$X$.
This gives an injective map 
\begin{equation}\label{eq:univ}
\phi\,: \, U(\lie)\, \hookrightarrow \, D_{n^2}.
\end{equation}

We now proceed to describing the algebra map $\,\phi\,$ explicitly. Fix an 
arbitrary element $\xi \in \lie$. Let $\,-M_\xi\,$ be the  $\,n \times n \,$ 
matrix corresponding to $\,\xi\,$ via the inclusion 
$\, \mathfrak{g} \hookrightarrow \mathfrak{gl}(n) $.
The following is the vector field encoding the 
Lie algebra action of $\,M_\xi\,$ on the space~$\,\mathfrak{gl}(n) \simeq \CC^{n \times n}$:
\begin{equation}
\label{eq:lieoperator} 
\phi(\xi) \,\, = \,\, \sum_{i,\, j \,= \, 1}^n (M_\xi)_{ij} \, \cdot \sum_{k=1}^n \, t_{jk} \partial_{ik}  \quad \in\,D_{n^2} .
\end{equation}

\begin{example}
Let $\,G\,=\, {\rm SO}(n)\,$ and  $\,\pi : G \, \to \, {\rm GL}_n(\RR)\,$
the standard representation on~$\RR^n$. The associated Lie algebra $\,\mathfrak{g}\,$ is
the space of skew-symmetric $\,n \times n\,$ matrices over~$\CC$.
A canonical basis of $\,\mathfrak{g}\,$
consists of the rank $\,2\,$ matrices $\,e_{ij}-e_{ji}\,$
for $ \,1 \leq i < j \leq n$.
The operator $\,P_{ij} \in D_{n^2}\,$ in Theorem~\ref{ann-theta}
is Fourier dual to  the vector field~\eqref{eq:lieoperator} if we take~$\,\xi \,=\,  e_{ji} - e_{ij}$.
\end{example}	

As seen in~\cite[Section 1.3]{htt},
the morphism of varieties $\,\pi : G_{\mathbb{C}} \rightarrow X\,$ induces a pushforward functor
of $D$-modules $\,\pi_+ : \text{Mod}(D_G)  \rightarrow \text{Mod}(D_{n^2})\,$
satisfying the following key property.

\begin{theorem}\label{thm:push}
If we regard $\,\CC[G]\,$ as a left $D_G\,$-module, then we have the isomorphism
\[\pi_+ ( \CC[G]) \,\, \cong \,\, D_{n^2}\, / \langle  \,I_G, \, \phi(\lie) \, \rangle.\]
In particular,  this quotient
 is a regular holonomic simple $D_{n^2}$-module.
\end{theorem}

\begin{proof}
By~\eqref{eq:tantriv}, we have the following isomorphism of right $D_G\,$-modules:
\begin{equation}\label{eq:struct}
\CC[G] \,\, \cong \,\, \CC \,\otimes_{U(\lie)}\, D_G.
\end{equation}
On the right, $\,\CC\,$ denotes the trivial representation of the universal enveloping algebra $U(\lie)$. 

Let $\,D_{G\to X} \, \coloneqq \, \CC[G]\,\otimes_{\CC[X]}\, D_{X}\,$ denote the transfer bimodule.
This is a left $D_G\,$-module and a right $D_X$-module. 
Since the action of $\,\lie\,$ extends to the whole space~$X$, 
we have $\,\CC[G]\, \cong \,\CC[X]/I_G\,$ as $\lie$-modules, 
and the left $U(\lie)$-structure of $\,D_{G\to X}\,$ is induced by the Leibniz rule 
via the map~\eqref{eq:univ} on the second factor. 
We obtain the isomorphism of bimodules
\begin{equation}\label{eq:trans}
D_{G\to X} \,\, \cong \,\,  D_X /( I_G \cdot D_X).
\end{equation}
 By (\ref{eq:struct}) and (\ref{eq:trans}), we have the following isomorphisms of right $D_{X}$-modules:
\begin{align*}\label{eq:isos}
\pi_+ (\CC[G]) &\, \coloneqq \,\,  \CC[G] \,\otimes_{D_G}\, D_{G\to X} \,\, \cong \,\,  (\CC\, \otimes_{U(\lie)}\, D_G) \, \, \otimes_{D_G} \,D_{G\to X} 
 \\&\,\,\cong \,\, \CC \, \otimes_{U(\lie)} \, D_{X}/(I_G \cdot  D_{X}) \,\, \cong \,\, D_{X}/((I_G+ \phi(\lie)) \cdot  D_{X}).
\end{align*}
The fist claim now follows by  switching to left $D_X$-modules. By
Kashiwara's Equivalence Theorem~\cite[Section 1.6]{htt}, the module
 $\,D_{X}/ \langle I_G, \phi(\lie) \rangle\,$ is regular holonomic and simple.
\end{proof}

\begin{remark}\label{rem:noncomp}
The assumption that $\,G\,$ is compact is not needed in Theorem~\ref{thm:push}.
The proof works for any representation
$\,\pi: H \to \GL_n(\CC)\,$ of a complex connected  algebraic group 
such that $\,\pi(H)\,$ is closed in $\CC^{n\times n}$. 
Such a representation exists for all  semi-simple groups~$H$.
Another natural setting is that of orbits of
a compact group $G$ acting linearly on a real vector space,
with left-invariant measures used in Corollary \ref{cor:anndist}.
In our view,
the theory of orbitopes~\cite{orbi} should be of interest for 
statistical inference with data sampled from orbits.
\end{remark}

\begin{remark} \label{rem:another}
Here is a more conceptual argument for Theorem \ref{thm:push}. 
The $D$-module \mbox{$\,M\,= \,D_{n^2}/ \langle I_G, \, \phi(\lie) \rangle\,$} 
is equivariant and supported on $\,\pi(G_{\CC})$ (see \cite[Section 11.5]{htt}). 
By Kashiwara's Equivalence Theorem, 
it is the pushforward of a coherent equivariant \mbox{$D$-module} on~$G_{\CC}$.
This is a direct sum of copies of the module~$\CC[G]$, by the Riemann--Hilbert Correspondence. 
Hence, $\,M\,$ is a direct sum of copies of $\,\pi_+ (\CC[G])$. 
The existence of a unique left-invariant 
measure on~$\,G\,$ implies that there is only one such summand in $M$.
\end{remark}

Let $\,\mu_\pi\,$ be the distribution on $\,\RR^{n\times n}\,$ given by integration against the Haar 
measure on~$G$. 

\begin{corollary}\label{cor:anndist}
The annihilator in $\,D_{n^2}\,$ of this distribution  equals
$$\ann_{D_{n^2}}(\mu_\pi) \,\,=\,\,  \langle \, I_G, \, \phi(\lie) \, \rangle.$$ 
\end{corollary}

\begin{proof}
Since ${\rm supp}(\mu_\pi)\,=\,\pi(G)$, we have $\,I_G \,\subset\, \ann_{D_{n^2}}(\mu_\pi)$. 
Since $\,\mu_\pi\,$ is a left-invariant distribution, we have also 
$\,\phi(\lie) \,\subset\, \ann_{D_{n^2}}(\mu_{\pi})$. 
By Theorem \ref{thm:push}, the $D$-ideal $\,\langle \, I_G, \, \phi(\lie) \, \rangle\,$ 
is a maximal left ideal in $D_{n^2},$
since its quotient is simple. It is therefore equal to $\ann_{D_{n^2}}(\mu_{\pi})$.
\end{proof}

The following observation establishes the connection to statistics,
as in~\cite[Section 4]{koyama}.

\begin{remark} \label{rmk:fourier}
The Fourier--Laplace transform of $\,\mu_\pi\,$ 
has a complex analytic continuation to a holomorphic function
on $\,\CC^{n \times n}\,$ by the Paley--Wiener--Schwartz Theorem, namely
\begin{equation}
\label{eq:normalizing3}
c(\Theta) \,\,= \,\,\int_{G} \exp(\textrm{tr}(\Theta^\text{t}\pi(Y)))\mu(dY).
\end{equation}
This is the normalizing constant of the Fisher distribution on the group $\pi(G)
\,\subset\, {\rm GL}_n(\RR)$. Note that this can be defined for a complex representation 
$\,\pi(G)\,\subset\, {\rm GL}_n(\CC)\,$ as well.
\end{remark}

The Fourier transform, denoted by~$(\bullet)^{\mathcal{F}},$ 
switches the operators $\,t_{ij}\,$ and $\,\partial_{ij}\,$
in the Weyl algebra~$D_{n^2}$, with a minus sign involved.
We consider the image of the $D$-ideal in
Corollary~\ref{cor:anndist} under this automorphism of~$D_{n^2}$.
This image is a $D$-ideal $J_\pi$ that is defined over~$\RR$:
\begin{equation}
\label{eq:jaypi}
 J_\pi \,\, = \,\, \langle \, I_G, \,\phi(\lie) \, \rangle^{\mathcal{F}} .
\end{equation}
 The following result
generalizes Theorem~\ref{ann-theta} to compact Lie groups other than $\SO(n)$.

\begin{corollary}\label{cor:annfun}
The $D$-module $\,D_{n^2}/J_\pi\,$ is simple holonomic
and $\,\ann_{D_{n^2}} (c(\Theta)) \,=\, J_\pi$.
\end{corollary}

\begin{proof}
By Corollary~\ref{cor:anndist}, Remark~\ref{rmk:fourier},
 and the defining property of the Fourier transform, 
 we see that $\,J_\pi\,$ annihilates 
 the integral in (\ref{eq:normalizing3}).
The proof  concludes  by recalling that the Fourier transform induces an 
 auto-equivalence on the category of (holonomic) $D_{n^2}$-modules.
 \end{proof}

We saw in  Section~\ref{sec5} that sampling from
$\,{\rm SO}(3)\,$ is ubiquitous in the applied sciences.
It would be worthwhile to explore such scenarios 
also for other matrix groups $\pi(G)$, and to apply
holonomic methods to
maximum likelihood estimation in their Fisher model.

\smallskip

One promising context for data applications is the unitary groups  in quantum physics.

\begin{example}\label{ex:su2}
The compact group $\,G=\operatorname{SU}(2)\,$ consists of 
complex $\,2 {\times} 2\,$ matrices of the~form
\begin{equation}
\label{eq:SU2}
\begin{pmatrix} \alpha & \beta \\ -\overline{\beta} & \overline{\alpha}  
\end{pmatrix}, \quad \mbox{ with } \,\,|\alpha|^2\,+\,|\beta|^2\,\, = \,\, 1.
\end{equation}
Note that $\,G\,$ is a double cover of $\,\SO(3)$. 
While the odd-dimensional (complex) representations of $\,G\,$ descend 
to real-valued representations of $\SO(3)$, 
this is not true for the even-dimensional (spin) representations. 
Consider the standard representation $\,G\,\subset\, \CC^{2\times 2}$.
 
The complexification of the matrix group in \eqref{eq:SU2} is simply the group
 $\SL_2(\CC) \,\subset\, \CC^{2\times 2}$.
The associated (maximal, holonomic) ideal $\,J_\pi\,$ is generated by the following four operators:
\begin{align*}
& d\,\, = \,\, \det(\partial)-1, & & h\,\, = \,\, t_{11}\partial_{11}+t_{12}\partial_{12} - t_{21}\partial_{21} - t_{22}\partial_{22},\\
& e \,\, = \,\,  t_{21}\partial_{11} + t_{22}\partial_{12}, & & f \,\, = \,\,  t_{11}\partial_{21} + t_{12}\partial_{22}.
\end{align*}
A computation shows that
$\,\rank J_\pi \,=\, 2\,$ and $\,\Sing (J_\pi)\,=\,\{\Theta \in \CC^{2\times 2}\, | \, \det(\Theta)\,=\,0\}$. 
The Lie algebra operators $\,e,f,h\,$ ensure that every holomorphic solution to 
$\,J_\pi\,$ is $\,\SL_2$-invariant.
By~\cite{luna},  
every solution has the form $ \Theta \,\mapsto\, \phi({\rm det}(\Theta))$,
for some analytic function $\,\phi\,$ in a domain of $\,\CC^*$. 
This  is annihilated by $\,d\,$ (hence, by $J_\pi$) if and only if $\,\phi(x)\,$ is annihilated by
\[x\partial^2 \,+\,2 \partial \,-\,1  \,\, \in \,\, D_1.\]
This has only one (up to scaling) entire solution $\phi$, with series expansion at $\,x\,=\,0\,$ given~by
\[ \phi(x) \,\, = \,\,  \sum_{n=0}^{\infty}\,\, \frac{1}{n! \cdot (n+1)!} \, \, x^n.\]
By comparing constant terms, we conclude that $\, c(\Theta) \, = \,\phi\bigl( {\rm det}(\Theta) \bigr)$.
It is straightforward to generalize the above considerations to the 
fundamental representation of the special unitary group $\text{SU}(m)$ 
for any $m \, \geq \, 1 $. In that setting, we find that~$\rank (J_{\pi})\,= \,m \,$.
\end{example}

In conclusion, the $D$-ideal $\,J_\pi\,$ is an interesting object that
deserves further study, not just for the rotation group ${\rm SO}(n)$,
but for arbitrary Lie groups $G$.
Sections \ref{sec:holonomic} and \ref{sec6} offer numerous suggestions for future research.
For instance, what is the holonomic rank of $J_\pi$?
Furthermore, it would be desirable to experiment with data sampled from groups $\,G\,$
other than ${\rm SO}(3)$, so as to broaden the
applicability of algebraic analysis in statistical inference.

\bigskip
\bigskip
\bigskip

\noindent {\bf Acknowledgments.}
We thank Mathias Drton for helpful discussions on statistics,  
Ralf Hielscher for discussions on materials science, and
Max Pfeffer for improvements in our numerical methods.
We are grateful to
Nobuki Takayama and his collaborators for many insightful discussions, 
and to Charles Wang for getting us started on the material for~$\text{SO}(n)$.

\bigskip \bigskip  \bigskip

\noindent
\footnotesize 
{\bf Authors' addresses:}

\smallskip

\noindent Michael F. Adamer, MPI-MiS Leipzig
\hfill {\tt michael.adamer@mis.mpg.de}

\noindent Andr\'{a}s C. L{\H o}rincz, MPI-MiS Leipzig
\hfill {\tt andras.lorincz@mis.mpg.de}

\noindent Anna-Laura Sattelberger,  MPI-MiS Leipzig
\hfill {\tt anna-laura.sattelberger@mis.mpg.de}

\noindent Bernd Sturmfels,
MPI-MiS Leipzig and UC Berkeley 
\hfill {\tt bernd@mis.mpg.de}


\bigskip

\begin{thebibliography}{10}

\begin{small}
\setlength{\itemsep}{-0.5mm}

\bibitem{dmodloc} D.~Andres:
{\em dmodloc\_lib: A Singular:Plural library for localization of algebraic $D$-modules and applications},
{\tt www.singular.uni-kl.de/Manual/latest/sing\_723.htm\#SEC775}

\bibitem{BHJPSW} F.~Bachmann, R.~Hielscher, P.\,E.~Jupp, W.~Pantleon, H.~Schaeben, and E.~Wegert: 
{\em Inferential statistics of electron backscatter diffraction data from within individual crystalline grains}, 
Jour. Appl. Crystallography \textbf{43(6)} (2010), 1338--1355.


\bibitem{BNV} M.\,A.~Bingham, D.\,J.~Nordman, and S.\,B.~Vardeman: 
{\em Bayes inference for a tractable new class of non-symmetric distributions for 3-dimensional rotations}, 
J.~Agric., Biol., and Environm. Stat. {\bf 17(4)} (2012), 527--543.

\bibitem{BNVII} M.\,A.~Bingham, D.\,J.~Nordman, and S.\,B.~Vardeman: 
{\em Finite-sample investigation of likelihood and Bayes inference for the symmetric von Mises--Fisher distribution}, 
Comp. Stat. \& Data Anal. \textbf{54(5)} (2010), 1317--1327.

\bibitem{sonpaper} M.~Brandt, J.~Bruce, T.~Brysiewicz, R.~Krone, and E.~Robeva:
{\em The degree of {\rm SO($n$)}},   
Combinatorial Algebraic Geometry, 207-224, Fields Inst. Commun.~{\bf 80}, Fields Inst.Res.Math.Sci.,~2017.

\bibitem{DT} J.\,R.~Davis and S.\,J.~Titus: 
{\em Modern methods of analysis for three-dimensional orientational data}, 
Jour. Struct. Geol. \textbf{96} (2017), 65--89.

\bibitem{DLM} T.~Downs, J.~Liebman, and W.~Mackay: 
{\em Statistical methods for vectorcardiogram orientations}, 
Proc. XIth Int. Symp. Vectorcardiography (1974), 216--222.

\bibitem{GLS} P. G\"{o}rlach, C.~Lehn, and A.-L.~Sattelberger:
{\em Algebraic analysis of the hypergeometric function ${_1F_1}$ of a matrix argument},
in preparation.

\bibitem{M2}
D.\,R.~Grayson and M.\,E.~Stillman:
\newblock {\em Macaulay2, a software system for research in algebraic geometry},
\newblock Available at \url{http://www.math.uiuc.edu/Macaulay2/}.

\bibitem{HNTT} H.~Hashiguchi, Y.~Numata, N.~Takayama, and A.~Takemura: 
{\em Holonomic gradient method for the distribution function of the largest root of a Wishart matrix}, 
Journal of Multivariate Analysis \textbf{117} (2013), 296--312.

\bibitem{git}
P.~Heinzner:
\newblock {\em Geometric invariant theory on {S}tein spaces},
\newblock Math. Ann. \textbf{289(4)} (1991), 631--662.

\bibitem{htt}
R. Hotta, K. Takeuchi, and T. Tanisaki:
 {\em {$D$}-modules, Perverse Sheaves, and Representation Theory},
  volume 236 of {\em Progress in Mathematics},
 Birkh\"{a}user, Boston, MA, 2008.

\bibitem{Kag} Y.\,Y.~Kagan: 
{\em Double-couple earthquake source: symmetry and rotation}, 
Geophys. Jour. Int. \textbf{194(2)} (2013), 1167--1179.

\bibitem{KM}
C.\,G.~Khatri and K.\,V.~Mardia: 
{\em The von Mises--Fisher matrix distribution in orientation statistics}, 
Jour. Roy. Stat. Soc. B \textbf{39(1)} (1977), 95--106.

\bibitem{CK} C.~Koutschan: 
{\em HolonomicFunctions: A Mathematica package for dealing with multivariate 
	holonomic functions, including closure properties, summation, and integration}.
Available at {\tt www3.risc.jku.at/research/combinat/software/ergosum/RISC/HolonomicFunctions.\\html}. 

\bibitem{koyama} T.\,Koyama: {\em The annihilating ideal of the Fisher integral},
{\tt arXiv:1503.05261}.

\bibitem{holorank} T.\,Koyama, H.\,Nakayama, K.\,Nishiyama, and N.\,Takayama: 
{\em The holonomic rank of the Fisher--Bingham system of differential equations}, 
J. Pure Appl. Algebra \textbf{218} (2014),~2060--2071. 

\bibitem{LMM} V.~Levandovskyy and J.~Mart\'{i}n-Morales:
{\em dmod\_lib: A Singular:Plural library for algorithms for algebraic $D$-modules}, 
{\tt www.singular.uni-kl.de/Manual/latest/sing\_535.htm\#SEC587}.

\bibitem{luna}
D.~Luna:
\newblock {\em Fonctions diff\'{e}rentiables invariantes sous l'op\'{e}ration d'un
  groupe r\'{e}ductif},
\newblock Ann. Inst. Fourier (Grenoble) \textbf{26(1)} (1976), 33--49.

\bibitem{MJ} K.\,V.~Mardia and P.\,E.~Jupp: 
{\em Directional Statistics}, 
John Wiley \& Sons, 2009.

\bibitem{muir} R.\,J.~Muirhead: 
{\em Aspects of Multivariate Statistical Theory}, 
Wiley Series in Probability and Mathematical Statistics, John Wiley \& Sons Inc, New York, 1982.

\bibitem{NocWri}
J.~Nocedal and S.~Wright, {\em Numerical Optimization}, Springer Science \& Business Media, 2006.

\bibitem{PrenII}
M.\,J.~Prentice: 
{\em Orientation statistics without parametric assumptions}, 
Jour. Roy. Stat. Soc. B \textbf{48(2)} (1986), 214--222.

\bibitem{procesi}
C. Procesi:
 {\em Lie Groups: An Approach Through Invariants and Representations},
Universitext, Springer, New York, 2007.

\bibitem{RRA}
D.~Rancourt, L.\,P.~Rivest, and J.~Asselin: 
{\em Using orientation statistics to investigate variations in human kinematics}, 
Jour. Roy. Stat. Soc. C \textbf{49(1)} (2000), 81--94.

\bibitem{orbi} R.~Sanyal, F.~Sottile, and B.~Sturmfels: 
{\em Orbitopes},
Mathematika {\bf 57} (2011), 275--314. 

\bibitem{satstu} A.-L.~Sattelberger and B.~Sturmfels:
{\em $D$-modules and holonomic functions},
{\tt arXiv:1910.01395}.

\bibitem{sei} T.~Sei, H.~Shibata, A.~Takemura, K.~Ohara, and N.~Takayama:
{\em Properties and applications of the Fisher distribution on the rotation group},
J.~Multivariate Analysis {\bf 116} (2013), 440--455.

\bibitem{Tak} N.~Takayama: 
{\em Gr\"obner bases for rings of differential operators and applications}, 
T.~Hibi (ed.): Gr\"obner Bases -- Statistics and Software Systems, 
Springer, Tokyo, 2013, 279--344.

\bibitem{hgmR} N.~Takayama, T.~Koyama, T.~Sei, H.~Nakayama, and K.~Nishiyama:
{\em hgm: An R package for the holonomic gradient method},  
{\tt https://cran.r-project.org/web/packages/hgm/hgm.pdf}.

\end{small}
\end{thebibliography}
\end{document}